\documentclass[12pt]{amsart}
\usepackage[top=1.2in, left=1.2in, right=1.2in, bottom=1.2in]{geometry}
\usepackage[utf8x]{inputenc}
\usepackage{amsfonts}
\usepackage{amssymb,epsfig}
\usepackage{amsmath}
\usepackage{delarray}
\usepackage{graphicx}
\usepackage{color}
\usepackage{tensor}

\newtheorem{thm}{Theorem}[section]
\newtheorem{theorem}[thm]{Theorem}

\newtheorem{lemma}[thm]{Lemma}

\newtheorem{definition}[thm]{Definition}
\theoremstyle{remark}
\newtheorem{remark}[thm]{Remark}
\numberwithin{equation}{section}
\newtheorem{example}[thm]{Example}

\newcommand{\set}[1]{\left\{#1\right\}}

\newcommand{\QED}{\hfill $\square$\vspace{2mm}}

\newcommand{\N}{\mathbb{N}}

\newcommand{\cL}{\mathcal{L}}

\begin{document}

\title{Normalization of singular contact forms and primitive 1-forms}

\author{Kai Jiang}
\address{Beijing international Center for mathematical research, Peking University}
\email{kai.jiang@bicmr.pku.edu.cn}
\author{Truong Hong Minh}
\address{Institut de Mathématiques de Toulouse, UMR5219, Université Toulouse 3}
\email{minhup@gmail.com}
\author{Nguyen Tien Zung}
\address{Institut de Mathématiques de Toulouse, UMR5219, Université Toulouse 3}
\email{tienzung.nguyen@math.univ-toulouse.fr}
\thanks{N.T. Zung is partially supported by a research consulting contract with the
Center for Geometry and Physics, Institute for Basic Science (South Korea)}

\begin{abstract}{%
A differential 1-form $\alpha$ on a manifold
of odd dimension $2n+1$, which satisfies the 
contact condition $\alpha \wedge (d\alpha)^n \neq 0$ almost everywhere, but which vanishes at a point $O$, i.e. $\alpha (O) = 0$, is called a \textit{singular contact form} at $O$. The aim of this paper is  to study local normal forms (formal, analytic and smooth) of such singular contact forms. Our study leads naturally to the study of normal forms
of singular primitive 1-forms of a symplectic form $\omega$
in dimension $2n$, i.e. differential 1-forms $\gamma$  which vanish at a point and such that $d\gamma = \omega$, and their corresponding conformal vector fields. Our results are an extension and improvement of previous results obtained by other authors, in particular Lychagin \cite{Lychagin-1stOrder1975},
Webster \cite{Webster-1stOrder1987} and Zhitomirskii \cite{Zhito-1Form1986,Zhito-1Form1992}. We make use of both the classical normalization techniques and the toric approach to the normalization problem for dynamical systems \cite{Zung_Birkhoff2005,Zung_Integrable2016,Zung_AA2018}.

}\end{abstract}

\date{First version, April 2018}

\maketitle

\section{Introduction}
Recall that a differential $1$-form $\alpha$ on a manifold $M$ 
over the field $\mathbb{K}$, where $\mathbb{K} = \mathbb{R}$
or $\mathbb{C}$, is called a \textbf{\textit{contact form}} if it is \emph{maximally nonintegrable} in the sense that $\alpha\wedge(d\alpha)^n\neq 0$ everywhere, where $1+2n=\dim M$ is the dimension of $M$. The corank-$1$ kernel distribution $\mathcal{D}=\ker\alpha$ on $M$ is called a contact structure in this case. 
According to the classical theorem of Darboux, all contact forms are locally isomorphic to the form $dx_0+\sum_{i=1}^nx_idx_{n+i}$, i.e. there is no local contact invariant, except for the dimension of the manifold 
(see, e.g., \cite{Geiges-Contact2004}). The situation becomes more interesting at singular points $O$, where $\alpha\wedge(d\alpha)^n(O)= 0$,
i.e., where $\alpha$ ceases to be contact. 

There are two types of singular points of degenerate contact forms:

\begin{enumerate}
\item Points $O$ such that $\alpha(O)\neq 0$, i.e. the kernel distribution $\mathcal{D}=\ker\alpha$ is still regular at $O$, but $\alpha\wedge(d\alpha)^n(O)=0$; 
\item Points $O$ such that $\alpha(O)=0$.
\end{enumerate}

Singular points of the first type have been extensively studied by many 
authors, such as Martinet, Roussarie, Pelletier, Zhitomirskii and Jakubczyk, 
see, e.g., \cite{JaZh-NF1995,JaZh-SingularContact2001, JaZh-Corank1-2003, Martinet-Formes1970, Pelletier-1Form1985, Russarie-Local1975, Zhito-Pfaff1989, Zhito-1Form1992}. In this paper, we are interested in singular points of the second type, i.e. points $O$ such that $\alpha(O)=0$. We will say that $\alpha$ is a \textit{\textbf{singular contact form}} in this case. Results on the local structure of singular contact 1-forms, and the corresponding singular first-order partial differential equations, have also been obtained by some authors, in particular Lychagin \cite{Lychagin-1stOrder1975},
Webster \cite{Webster-1stOrder1987} and Zhitomirskii \cite{Zhito-1Form1986,Zhito-1Form1992}. The aim of our paper is to make a deeper study and improve these previous results. 

We will say that $O$ is a \textbf{\textit{nondegenerate singular point}}  of a singular contact form $\alpha$ in $2n+1$
dimensions
if it satisfies the following two nondegeneracy conditions:
\begin{itemize}
\item  $d\alpha$ is a regular presymplectic form of rank $2n$ and corank 1 near $O$, i.e.
\begin{equation} 
\label{eqn:alpha-n}
(d\alpha)^n(O)\neq 0.
\end{equation}
\item Denote by $f_i(x)$ the coefficients of $\alpha$
in a  local coordinate system  $(x_0,\ldots,x_{2n})$ near $O$, $\alpha=\sum_{i=0}^{2n}f_i(x)dx_i$, then
\begin{equation}\label{eqn:alphaF}
F=(f_0,\ldots,f_{2n}): (M^{2n+1}, O)\rightarrow (\mathbb{K}^{2n+1}, 0) \ \text{is a local diffeomorphism}.
\end{equation}
\end{itemize}

Clearly, the above nondegeneracy condition is a generic and stable condition: any small perturbation of $\alpha$ will also admit a unique nondegenerate singular point (where it vanishes) near $O$. Notice also that Conditions \eqref{eqn:alpha-n} and \eqref{eqn:alphaF} are independent and do not imply each other. For example,
$\sum_{i=0}^{2n} x_i dx_i$ satisfies \eqref{eqn:alphaF} but does not satisfy \eqref{eqn:alpha-n}. On the other hand,
$x_1dx_1 + \sum_{i=1}^{n} x_i dx_{i+n}$ 
satisfies \eqref{eqn:alpha-n} but does not satisfy \eqref{eqn:alphaF}. All singular points of singular contact
forms in this paper will be assumed to be nondegenerate. 

It follows from Equation \eqref{eqn:alpha-n} that the kernel $\ker(d\alpha)(x)$ of $d\alpha$ is of dimension $1$ for all $x$ near $O$. 
Let $Z$ be a local nowhere-vanishing vector field which generates $\ker(d\alpha)$, i.e., $\ker(d\alpha)(x)=\mathbb{K}.Z(x)$ for all $x$ 
near $O$. By the classical Darboux theorem, there is a (smooth or analytic) coordinate system $(\theta, x_1,\ldots, x_{2n})$ in a neighborhood $\mathcal{U}(O)$
of $O$ such that the components of $Z$ satisfies $Z_\theta\neq0, Z_{x_i}=0 $ for every $i$, and the presymplectic form
$d\alpha$ has the following canonical form:
\begin{equation}
d\alpha=\sum_{i=1}^n dx_i\wedge dx_{n+i}.
\end{equation} 
In particular, we have
$\alpha\wedge(d\alpha)^n=(-1)^{\frac{(n-1)n}{2}}n! \langle \alpha, \dfrac{\partial}{\partial \theta}\rangle d\theta\wedge\bigwedge_{i=1}^{2n}dx_i.
$
Due to the nondegeneracy condition \eqref{eqn:alphaF}, the function $f_\theta = \langle \alpha, \dfrac{\partial}{\partial \theta}\rangle$ (which is just a coefficient function of $\alpha$) is a regular function on $\mathcal{U}(O)$ which vanishes at $O$,
hence the set $N$ of non-contact points in  
$\mathcal{U}(O)$, i.e. 
\begin{equation}
N=\{x\in\mathcal{U}(O):\alpha(Z)(x)=0\}
\end{equation}
is a regular hypersurface. Now we have two possibilities: 
either the kernel of $d\alpha$ at $O$ (which is generated
by $Z := \dfrac{\partial}{\partial \theta}$) is transversal
to $N$ at $O$, or it is tangent, so we make the following definition:

\begin{definition}
A point  $O$  is called a \textbf{transversal nondegenerate singularity} of a singular contact 
form $\alpha$ if $O$ is a nondegenerate singularity of $\alpha$ and the set $N$ of non-contact points is transverse to the kernel  $\ker(d\alpha)(O)$;
$O$ is called a \textbf{tangent nondegenerate singularity} if $O$ is nondegenerate and $N$ is tangent to the kernel $\ker(d\alpha)(O)$.
\end{definition}

Both of the above cases can be easily realized, as in the
following examples in $\mathbb{K}^3$ with a coordinate system $(\theta, x_1, x_2)$:
\begin{itemize}
\item[1)] $\alpha=\theta d\theta+\frac{1}{2}(x_1dx_2-x_2dx_1)$, $N=\{\theta = 0\}$ which is transverse to $Z=\dfrac{\partial}{\partial\theta}$.
\item[2)] $\alpha=x_1d\theta+(\theta-x_2)dx_1$,  $N=\{x_1=0\}$ which is tangent to 
$Z=\dfrac{\partial}{\partial\theta}$.
\end{itemize}

One of the main results of our paper is the following 
pre-normalization theorem:

\begin{theorem} \label{thm:prenormalization}
Let $O$ be a nondegenerate singularity of a smooth (resp., real or complex analytic) singular contact 1-form $\alpha$ on a manifold
of dimension $2n+1$. Then there is a local smooth (resp., analytic)
coordinate system $(\theta, x_1, \hdots, x_{2n})$ in which
$\alpha$ has the expression
\begin{eqnarray} \label{eqn:alphaTransversal}
\alpha = \theta d \theta + \gamma   
\end{eqnarray}
in the transversal case, or the expression
\begin{eqnarray} \label{eqn:alphaTangent}
\alpha = d(\theta^3 - x_1\theta) + \gamma   
\end{eqnarray}
in the tangent case with a generic tangency, 
where $\gamma = \sum_{i=1}^{2n} g_i d x_i$
is a 1-form which is basic with respect to  $\dfrac{\partial}{\partial\theta}$, i.e. the functions $g_i$ do not depend on
$\theta$, and such that 
\begin{eqnarray} 
d\gamma = \sum_{i=1}^n d x_i \wedge d x_{i+n}   
\end{eqnarray}
is a symplectic form in $2n$ variables (which can be put into Darboux canonical form).
\end{theorem}

We remark that the transversal case of the above theorem was already obtained by Zhitomirskii \cite{Zhito-1Form1992}, but the tangent case is new, as far as we know. Theorem \ref{thm:prenormalization}  remains true if we replace the nondegeneracy condition 
\eqref{eqn:alphaF} by a weaker condition: the map $F$ is not required to be a local diffeomorphism; only one of the coefficients of $\alpha$, namely the one corresponding to the kernel of $d\alpha$, is required to be a regular function. (That's what we used in the proof of the theorem). 

Theorem \ref{thm:prenormalization} allows us to reduce the problem of local normalization of singular contact forms (under a nondegeneracy condition)
to the problem of local normalization of 
1-forms $\gamma$ such
that $\omega = d\gamma$ is a symplectic form. We will call such a form 
$\gamma$ a \textbf{\textit{primitive form}} (of the symplectic form $\omega =d\gamma$, which can be assumed to be canonical, due to Darboux theorem). This problem will be studied in Section \ref{section:primitive} of this paper, via the normalization of the associated vector field 
 $X$ defined by:
\begin{equation}
i_X \omega = \gamma,
\end{equation}
which is a \textit{\textbf{conformal vector field}} of $\omega$
in the sense that
\begin{equation}
\mathcal{L}_X \omega = \omega. 
\end{equation}
Applying the theory of formal, analytic and smooth normalizations of vector fields to $X$, we obtain a (formal,
smooth or analytic) canonical system of coordinates
$(x_1,\hdots,x_{2n})$ which normalizes both $\omega$ and
$X$ simultaneously, i.e. the symplectic form is 
\begin{equation}
\omega = \sum_i dx_i \wedge dx_{n+i}
\end{equation}
and the Taylor power series of $X$ in this coordinate system can be written as
\begin{equation}
X = X^s + X^n + \sum_{k \geq 2} X^{(k)},
\end{equation}
where $X^s$ is the semisimple part of the linear part of $X$ in the normal form (hence $X^s$ is also the semisimple part of the whole $X$), 
$X^n$ is the nilpotent part of the linear part of $X$ in the normal form, and  $X^{(k)}$ is the resonant part of degree $k$ of $X$: $[X^s, X^{(k)}] = 0$ for all $k \geq 2$. The fact that
that $\mathcal{L}_X \omega = \omega$ also implies the eigenvalues
of $X$ come in pairs of numbers $(\lambda_i, 1-\lambda_i)$ 
whose sum is equal to 1 (similarly to the Hamiltonian case,
where the eigenvalues come in pairs of opposite numbers).
By contracting $X$ with $\omega$ in normal form, we get the following theorem
on the normalization of $\gamma$:

\begin{theorem} \label{thm:PrimitiveNF}
Let $\gamma$ be a smooth or analytic primitive form of a symplectic form $\omega$ in $2n$ dimensions, i.e., $d\gamma = \omega$, such that $\gamma$ vanishes at a point $O$. Denote by $X$ the associated conformal vector field, i.e. $i_X \omega = \gamma$. 
Then we have:

i) (Formal normalization) $\gamma$ can be formally normalized (over $\mathbb{C}$) such that $\omega$ is the standard symplectic form and $X$ is in the Poincar\'e-Dulac normal form with diagonal semisimple part $X^s$. 

It follows that $X^s$ is also a conformal vector field of $\omega$ in the sense that $\cL_{X^s}\omega=\omega$ and the linear nilpotent part $X^n$ together with the higher order part $X^{(k)}$ for $k\geq2$ are (degenerate) Hamiltonian vector fields of $\omega$, and the corresponding Hamiltonian functions are conformally preserved by $X^s$.

Concretely, $\gamma$ has an explicit expression as follows: there are positive integers $n_1=1<n_2<\cdots<n_{k}<n_{k+1}=n+1$, and eigenvalues $\lambda_1, \hdots, \lambda_{2n}$ of $X$, such that $\lambda_i + \lambda_{n+i} = 1$ for all $i =1,\hdots, n$,
$\lambda_i = \lambda_j$ for any $n_s \leq i < j < n_{s+1}$ ($s=1,\hdots, k)$, and
\begin{equation}
\gamma=\sum_{i=1}^k (\gamma_i +dQ_i)+dR;
\end{equation}
where 
\begin{equation}
\gamma_i=\sum_{j=n_i}^{n_{i+1}-1}\left(\lambda_{n_i} x_{j}dx_{n+j}+(\lambda_{n_i}-1) x_{n+j}dx_{j}\right);
\end{equation}
and 
\begin{itemize}
\item{} If $\lambda_i\neq\frac{1}{2}$ then $Q_i=0$ or $Q_i=\sum_{j=n_i}^{n_{i+1}-2}x_{j+1}x_{n+j}$;
\item{} If $\lambda_i=\frac{1}{2}$ then $Q_i$ belongs to one of the following four cases:
	\begin{itemize}
    \item{} $Q_i=0$;
	\item{} $Q_i=x_{n_i+n}^2$ with $n_{i+1}=n_i+1$ in this case;
	\item{} $Q_i=2\sum_{j=n_i}^{n_{i+1}-2}x_jx_{n+j+1}+(-1)^{n_{i+1}-n_i}x_{n_{i+1}-1}^2$;
	\item{} $Q_i=2\sum_{j=n_i}^{n_{i+1}-2}x_jx_{n+j+1}$ with $n_{i+1}-n_i\geq 3$ and is an odd number;
	\end{itemize}
\end{itemize}
and $R$ is a function of $x_1,\ldots,x_{2n}$ whose infinite jet at $0$ contains only the monomial terms (up to constant coefficients) $\prod x_i^{\alpha_i}$ satisfying a resonance relation
\begin{equation}
\label{ResonantRelation}
\sum_{i=1}^{2n}\alpha_i \lambda_i = 1 \quad \text{with}\;
\sum_{i=1}^{2n}\alpha_i\geq 3.
\end{equation}

ii) (Analytic normalization) If $\gamma$ is analytic and the eigenvalues of $X$ satisfies a Diophatine condition, e.g. the Bruno condition (see Definition \ref{def:Bruno}), then $\gamma$ can be analytically normalized to the above normal form.

iii) (Smooth normalization) If $\gamma$ is smooth and $X$ is hyperbolic, i.e., none of its eigenvalues $\lambda_1,\ldots,\lambda_n$,$1-\lambda_1, \ldots, 1-\lambda_n$ lies on the imaginary axis, then $\gamma$ is smoothly normalizable.

iv) (Linearization) In particular, if there are no resonance relation, i.e., there does not exist any $2n$-tuple of nonnegative integers $(\alpha_1,\ldots,\alpha_{2n})$ satisfying \eqref{ResonantRelation}, then $\gamma$ is smoothly linearizable, or analytically linearizable under the Bruno condition.
\end{theorem}

To get normal forms of singular contact forms, one simply 
puts Theorem \ref{thm:prenormalization} and Theorem \ref{thm:PrimitiveNF} together. Notice that, when the primitive 1-form $\gamma$ in the 
pre-normal form  \eqref{eqn:alphaTangent} in the tangent case is normalized, the old function $x_1$ 
in \eqref{eqn:alphaTangent} becomes a regular function 
$\phi(x_1,\hdots,x_{2n})$ in the new normalized coordinate system for
$\gamma$, which cannot be made into a linear function in these new
coordinates in general, because it is unrelated to $\gamma$. 

\medskip

The rest of this paper is organized as follows: In Section 2 we study the normalization of singular primitive 1-forms by applying the classical methods of normalization, and also the toric approach 
to the problem of normalization of vector fields, to the associated conformal vector fields of these 1-forms. Theorem \ref{thm:PrimitiveNF}
is proved in this section. In Section 3 we reduce the problem of normalization of singular contact forms to the problem of normalization of
primitive 1-forms for nondegenerate singularities, both in the transversal case and the tangent case, and prove the pre-normalization theorem (Theorem \ref{thm:prenormalization}). This theorem together with Theorem \ref{thm:PrimitiveNF} give us normalization and linearization of singlar contact forms.

\section{Normalization of primitive 1-forms}
\label{section:primitive}
 
\subsection{Preliminaries}

Let $\gamma$ be a local primitive 1-form on $\mathbb{K}^{2n}$, so that $d\gamma=\omega$ is a symplectic form. 
Since $\omega$ is nondegenerate, there exits a unique vector field $X$ such that $\gamma=i_X\omega$. It follows from Cartan's formula that
\begin{equation}
\cL_X\omega=i_Xd\omega+di_X\omega=d\gamma=\omega.
\end{equation}
Because of this equation, we say that $X$ a \textit{\textbf{conformal vector field}} or \textit{\textbf{Liouville vector field} }  of $\omega$.
Assume that the origin $O$ of $\mathbb{K}^{2n}$ is a singular point of $\gamma$, i.e., $\gamma(O)=0$. Then $X(O) = 0$, and one can talk about the normalization a la Poincare-Birkhoff of $X$ at 0 as follows. Viewing $X$ as a linear operator acting by derivation on the space of formal functions (i.e., power series), we can decompose $X$ into the sum of its semisimple part and its nilpotent part, just like the Jordan decomposition for finite-dimensional endomorphisms:
\begin{equation}
X = X^S + X^N, \quad [X^S, X^N] = 0,
\end{equation}
and $X^S$ is (formally) diagonalizable over $\mathbb{C}$: there exists a
new formal coordinate system $(x_1,\hdots,x_{2n})$ in which the Taylor series of $X$ is
\begin{equation}
X = X^s + X^n + \sum_{k \geq 2} X^{(k)}
\end{equation}
where $X^{(k)}$ consists of homogeneous terms of degree $k$, $X^n$ is linear nilpotent, 
\begin{equation}
X^S = X^s = \sum_{i=1}^{2n} \lambda_i x_i \dfrac{\partial}{\partial x_i}
\end{equation}
is a diagonal vector field and $X^N = X^n + \sum_{k \geq 2} X^{(k)}$. The minimal number $\tau$ such that $ X^s$ can be written as
\begin{equation}
X^s = \sum_{i=1}^{\tau} c_i Z_i,
\end{equation}
with 
\begin{equation}
Z_i = \sum_{j=1}^{2n} \rho_{ij} x_j \dfrac{\partial}{\partial x_j}
\end{equation}
being diagonal vector fields with integer coefficients ($\rho_{ij} \in \mathbb{Z}$), is called the \textbf{\textit{toric degree}} of $X$ at $O$.
 The fact that $\tau$ is minimal is equivalent to the fact that the numbers $c_1,\hdots, c_\tau$ are incommensurable. The vector fields 
 $\sqrt{-1}Z_1, \hdots, \sqrt{-1}Z_\tau$ generate a (formal) effective
 torus $\mathbb{T}^\tau$-action (in the complexified space if $X$ is real),
 which is intrinsic (i.e., it depends only on $X$ and not on the choice of
 coordinates), unique up to automorphisms of $\mathbb{T}^\tau$, and is called the \textit{\textbf{associated torus action}} of $X$ at $O$. The linearization of this associated torus action is equivalent to the normalization \`a la Poicar\'e-Birkhoff.
(See \cite{Zung_Birkhoff2005,Zung_Integrable2016,Zung_AA2018} and references therein). A general \textit{conservation law} studied by Zung 
\cite{Zung_AA2018} says that ``anything'' (conformally) preserved by a 
dynamical system is automatically (conformally) 
preserved by its intrinsic associated torus actions. The following lemma is just a particular case of this general conservation law:

\begin{lemma}
\label{lem2.1}
Let $X^{S}$ be the (intrinsic) semisimple part of $X$. If $\omega$ is any formal or analytic differential 2-form such that $\cL_{X}\omega=\omega$ then $\cL_{X^{S}}\omega=\omega$. 
\end{lemma}

\begin{proof}
For the self-containment, let us provide a proof here. We can suppose that $X$ is in normal form, i.e., its semisimple part $X^{S}$ is also the semisimple part of of its linear part $X^{s}$, and $X^{s}$ is written as $X^s=\sum_i\lambda_ix_i\frac{\partial}{\partial x_i}$ in a coordinate system $(x_1,\ldots,x_{2n})$.
Decompose $X$ and $\omega$ into homogeneous terms, using the following notation:
\begin{equation}
X=X^{(1)}+X^{(2)}+\ldots,\quad
\omega =\omega_0+\omega_1+\ldots.
\end{equation}
We must have $\cL_{X^{(1)}}\omega_0=\omega_0$, which implies that $\cL_{X^s}\omega_0=\omega_0$. Suppose that $\cL_{X^s}\omega_i=\omega_i, \forall i=0,\ldots,m-1$. 
The condition $\cL_{X}\omega=\omega$ follows that
\begin{equation}\label{eq2.8}
\cL_{X^{(1)}}\omega_m+\cL_{X^{(2)}}\omega_{m-1}+\ldots+\cL_{X^{(m+1)}}\omega_0=\omega_m.
\end{equation}
Denote by $\eta=\cL_{X^{(2)}}\omega_{m-1}+\ldots+\cL_{X^{(m+1)}}\omega_0$, then $\cL_{X^s}\eta=\eta$. Taking the Lie derivative $\cL_{X^s}$ on both sides of \eqref{eq2.8}, we get
\begin{equation}\label{eq2.9}
\cL_{X^{(1)}}\cL_{X^s}\omega_m+\eta=\cL_{X^s}\omega_m
\end{equation}
Equalities \eqref{eq2.8} and \eqref{eq2.9} lead to
\begin{equation}
\cL_{X^{(1)}}\left(\cL_{X^s}\omega_m-\omega_m\right)=\cL_{X^s}\omega_m-\omega_m.
\end{equation}
Consequently,
\begin{equation}\label{eq2.11}
\cL_{X^s}\left(\cL_{X^s}\omega_m-\omega_m\right)=\cL_{X^s}\omega_m-\omega_m.
\end{equation}
Suppose that that $\omega_{m}^j$ is a monomial term of $\omega_m$ then $\cL_{X^s}\omega_m^j=c_j\omega_m^j$ for some constant $c_j$. Equality \eqref{eq2.11} forces $c_j=1$, and this implies that $\cL_{X^s}\omega_m=\omega_m$.
\end{proof}

\begin{lemma}[Equivariant Darboux theorem]
\label{lem:Darboux} 
Suppose that $X$ is in Poincar\'e-Dulac (formal or analytic) normal form and $\omega$ is a symplectic form which is conformally preserved by $X$: $\mathcal{L}_X \omega = \omega$. Then, there is a local
(formal or analytic) diffeomorphism $\phi$ which preserves $X^s$, i.e. $\phi_*X^s=X^s$ and conjugates $\omega$ with the constant two-form $\omega_0=\omega(O)$.
\end{lemma}

\begin{proof}
We first show that the diffeomorphism in Moser's path method preserves the semisimple part $X^s$. Indeed, let us 
denote by $\zeta=i_{X^s}(\omega-\omega_0)$ then 
\begin{equation}
d\zeta=\cL_{X^s}(\omega-\omega_0)-i_{X^s}d(\omega-\omega_0)=\omega-\omega_0.
\end{equation}
Now, consider the path
$\omega_t=t\omega+(1-t)\omega_0$ ($t \in [0,1]$), 
and let $Y(t)$ be the time dependent vector field defined by $i_{Y_t}\omega_t=-\zeta$. Then we have
\begin{equation}
\cL_{Y_t}\omega_t=-d\zeta=-(\omega-\omega_0)=-\frac{\partial \omega_t}{\partial t}.
\end{equation}
It implies that $\omega$ and $\omega_0$ are conjugated by the diffeomorphism which is the integration of $Y_t$. We need to prove that $Y_t$ commutes with $X^s$. Indeed, we have
\begin{align*}
\cL_{X^s}\left(i_{Y_t}\omega_t\right)=-\cL_{X^s}\zeta&=-i_{X^s}d\zeta-di_{X^s}\zeta\\
&=-i_{X^s}(\omega-\omega_0)-di_{X^s}i_{X^s}(\omega-\omega_0)=-\zeta.
\end{align*}
On the other hand, we also have
\begin{align*}
\cL_{X^s}\left(i_{Y_t}\omega_t\right)&=i_{[X^s,Y_t]}\omega_t+i_{Y_t}\cL_{X^s}\omega_t\\
&=i_{[X^s,Y_t]}\omega_t+i_{Y_t}\omega_t=i_{[X^s,Y_t]}\omega_t-\zeta.
\end{align*}
This leads to $i_{[X^s,Y_t]}\omega_t=0$. Consequently, $[X^s,Y_t]=0$.
\end{proof}

\subsection{Formal normal form of a primitive $1$-form}
We now establish a formal normal form of a primitive 1-form with the above preparations.

Suppose that in a normalizing formal complex coordinate system $(x_1,\ldots,x_{2n})$ we have that $X^{(1)}$ is in the Jordan normal form, $X^s=\sum_{i=1}^{2n}\lambda_i x_i\frac{\partial}{\partial x_i}$, $[X^s,X]=0$ and $\omega=\sum_{1\leq i<j\leq 2n}c_{ij}dx_i\wedge dx_j$. Since $\cL_{X^s}\omega=\omega$ and 
$$\cL_{\sum_{i=1}^{2n}\lambda_ix_i\frac{\partial}{\partial x_i}}\left(dx_i\wedge dx_j\right)=(\lambda_i+\lambda_j)dx_i\wedge dx_j,$$ 
we must have $\lambda_i+\lambda_j=1$ whenever $c_{ij}\neq 0$. Let us denote by 
$$ [\lambda]=\{\text{indexes } i \text{ such that } \lambda_i=\lambda\},\; ([\lambda]\subset\{1,\ldots,2n\}),$$
$$[[\lambda]]=\{\text{index pairs } (i,j) \text{ such that } i\in[\lambda], j\in[1-\lambda] \text{ or } 
j\in[\lambda], i\in[1-\lambda]  \}.$$
If $\omega$ contains a term $dx_i\wedge dx_j$ (i.e. $c_{ij}\neq 0$) then $(i,j)$ must be in  $[[\lambda]]$ for some $\lambda$. 
Hence, we can write
\begin{equation}
\omega=\sum_{k}\omega_{[[\lambda_k]]},
\end{equation}
where $\omega_{[[\lambda_k]]}$ contains only terms $dx_i\wedge dx_j$ such that $(i,j)\in [[\lambda_k]]$. Since $\omega$ is nondegenerate, $\omega_{[[\lambda_k]]}$ is also nondegenerate when we restrict it to the vector space $\mathrm{Span}\{\frac{\partial}{\partial x_i} | i\in[\lambda_k]\cup[1-\lambda_k]\}$. This allows us to reduce the study to the case with only one block $[[\lambda]]$. 

We first consider the case $\lambda\neq\frac{1}{2}$.
The two vector spaces $\mathrm{Span}\{\frac{\partial}{\partial x_i} | i\in [\lambda]\}$ and $\mathrm{Span}\{\frac{\partial}{\partial x_j}|j\in[1-\lambda]\}$ are both isotropic with respect to
\begin{equation}
\omega_{[[\lambda]]}=\sum_{(i,j)\in[[\lambda]]}c_{ij}dx_i\wedge dx_j.
\end{equation}
It follows that the cardinals of $[\lambda]$ and $[1-\lambda]$ are equal. Consequently, 
\begin{equation}
X^s_{[[\lambda]]}=\lambda E_{[\lambda]}+(1-\lambda)E_{[1-\lambda]},
\end{equation}
where $E_{[\lambda]}=\sum_{i\in[\lambda]}x_i\frac{\partial}{\partial x_i}$ and $E_{[1-\lambda]}=\sum_{j\in[1-\lambda]}x_j\frac{\partial}{\partial x_j}$ are the Euler vector fields on $\mathrm{Span}\{\frac{\partial}{\partial x_i} | i\in [\lambda]\}$ and $\mathrm{Span}\{\frac{\partial}{\partial x_j}|j\in[1-\lambda]\}$ respectively. Suppose that $[\lambda]=\{1,\ldots,s\}$ and $[1-\lambda]=\{s+1,\ldots,2s\}$. We can fix the coordinates $x_1,\ldots,x_s$ and change $x_{s+1},\ldots,x_{2s}$ by linear combinations such that 
\begin{equation}
\omega_{[[\lambda]]}(\frac{\partial}{\partial x_i},\frac{\partial}{\partial x_{j+s}})=\delta_{ij},\,\forall i,j=1,\ldots s.
\end{equation}
In this new coordinate system, we have
\begin{equation}
\omega_{[[\lambda]]}=\sum_{i=1}^s dx_i\wedge dx_j,
\end{equation}
\begin{equation}
X^s_{[[\lambda]]}=\lambda \sum_{i=1}^sx_i\frac{\partial}{\partial x_i}+(1-\lambda) \sum_{i=1}^sx_{i+s}\frac{\partial}{\partial x_{i+s}}.
\end{equation}
Denote by $X^{nil}$ the linear nilpotent terms and $X^{nil}_{[[\lambda]]}$ its restriction to the block $[[\lambda]]$. Then by the construction of the coordinates $(x_1,\ldots,x_{2s})$, we can write $X^{nil}_{[[\lambda]]}$ as follows: there exist positive integers $t_1=1<t_2<\cdots<t_{\ell+1} \leq s$ such that
\begin{equation}
\label{nil}
X^{nil}_{[[\lambda]]}=\sum_\ell\sum_{i=t_\ell}^{t_{\ell+1}-1}x_{i+1}\frac{\partial}{\partial x_i}+\sum_{i,j=s+1}^{2s}e_{ij}x_j\frac{\partial}{\partial x_i},\; 1\leq t\leq s-1,
\end{equation}
where $e_{ij}$ are constant coefficients such that the $s$ by $s$ matrix $(e_{ij})$ is nilpotent.

Assume $Y$ is a homogeneous vector field such that 
$\cL_{X^s}Y=0$ and $\cL_{Y}\omega=0$, then $i_Y\omega=-dH$ for some Hamiltonian function $H$. We have 
\begin{equation}
\cL_{X^s}(-dH)=\cL_{X^s}(i_Y\omega)=i_{[X^s,Y]}\omega + i_Y(\cL_{X^s}\omega)=-dH.
\end{equation}
This implies that $X^s(H)=H$. Therefore, $H$ contains only the monomials $\prod x_i^{\alpha_i}$ satisfying the following resonance relation
\begin{equation}
\sum_{i=1}^{2n}\lambda_i\alpha_i=1.
\end{equation}
By this observation, we must have 
\begin{equation}
i_{X^{nil}_{[[\lambda]]}}\omega_{[[\lambda]]}=-dQ
\end{equation}
for some quadratic function $Q$. Moreover, by simple computations we obtain that for quadratic monomials $Q_{ij}=x_ix_j,\, Q_{i,j+s}=x_ix_{j+s},\, Q_{i+s,j+s}=x_{i+j}x_{j+s}$ for $1\leq i,j\leq s$, their corresponding vector fields read respectively
\begin{equation}
\begin{aligned}
&X_{x_ix_j}=x_i\frac{\partial}{\partial x_{j+s}}+x_j\frac{\partial}{\partial x_{i+s}},\\
&X_{x_{i+s}x_{j+j}}=-x_{i+s}\frac{\partial}{\partial x_{j}}-x_{j+s}\frac{\partial}{\partial x_{i}},\\
&X_{x_ix_{j+s}}=-x_i\frac{\partial}{\partial x_{j}}+x_{j+s}\frac{\partial}{\partial x_{i+s}}.
\end{aligned}
\end{equation}
When $\lambda\neq\frac{1}{2}$, the quadratic monomials like $Q_{i,j+s}$ is the only choice among the three types due to the resonance relation.  
Comparing to \eqref{nil}, we get 
\begin{equation}
Q=\sum_{i=t_\ell}^{t_{\ell+1}-1} x_{i+1}x_{s+i}\quad \text{and} \quad 
X^{nil}_{[[\lambda]]}=\sum_{i=t_\ell}^{t_{\ell+1}-1}x_{i+1}\frac{\partial}{\partial x_i}-x_{s+i}\frac{\partial}{\partial x_{s+i+1}}
\end{equation}
in each sub-block.

In the case $\lambda=\frac{1}{2}$, the cardinal of $[\frac{1}{2}]$ must be even, still assumed to be $2s$. Moreover, we have
\begin{equation}
X^s_{[\frac{1}{2}]}=\frac{1}{2}E_{[\frac{1}{2}]}=\frac{1}{2}\sum_{i\in {[\frac{1}{2}]}}x_i\frac{\partial}{\partial x_i}.
\end{equation}
By a linear change of coordinates, $\omega_{[\frac{1}{2}]}$ can be put into the canonical form in this block and $X^s_{[\frac{1}{2}]}$ is still equal to   $\frac{1}{2}E_{[\frac{1}{2}]}$. Still by comparing the corresponding vector fields of the three types of quadratic monomials with \eqref{nil}, we conclude that such a quadratic function should be in one of the following forms:
\begin{equation}
\begin{aligned}
     &Q_{t_\ell}=x_{t_\ell+s}^2 \text{ with } t_{\ell+1}=t_\ell+1;\\
	 &Q_{t_\ell}=2\sum_{j=t_{\ell}}^{t_{\ell+1}-2}x_jx_{s+j+1}+(-1)^{t_{\ell+1}-t_\ell}x_{t_{\ell+1}-1}^2;\\
	 &Q_{t_\ell}=2\sum_{j=t_\ell}^{t_{\ell+1}-2}x_jx_{s+j+1} \text{ with } t_{\ell+1}-t_\ell\geq 3 \text{ and is an odd number}.
\end{aligned}
\end{equation}


\QED

\subsection{Analytic normal form}
If the singular contact form is in the analytic category, a natural but difficult question is whether  the formal normalization converges. 
Notice that the Moser's path method preserves the analyticity, so we can immediately conclude that if $X$ has a convergent transformation to its normal form, then any analytic primitive 1-form admitting $X$ as the associated conformal vector field 
is also analytically normalizable. One can refer to \cite{Walcher} for a survey which contains the main criteria for the convergence problem.  

Here we now recall two conditions named after Bruno on analytic vector fields.
\begin{definition}
\label{def:Bruno}
Suppose $X$ is an analytic vector field on $\mathbb C^n$ having $0$ as an equilibrium point. Denote by $X^{(1)}$ its linear part. We say it satisfies 
\begin{itemize}
 \item\textbf{condition} $\boldsymbol\omega$, if the eigenvalues $(\lambda_1,\ldots,\lambda_n)$ of $X^{(1)}$ satisfies 
$$\sum_{k=1}^\infty 2^{-k}\ln\frac{1}{\omega_k}<\infty \quad\mbox{where}\quad \omega_k:=\min_{\substack{|\ell_1+\cdots+\ell_n|<2^k\\r=1,\ldots,n}}\set{|\sum_{s=1}^n\ell_s\lambda_s-\lambda_r|\neq0:\ell_s\in\N}$$
 \item \textbf{condition} $\mathbf A$, if it has a formal normal form like $f\cdot X^{(1)}$ for some function $f$
\end{itemize} 
\end{definition}

Bruno proved (see {\cite{Bruno_Book1989}}) that if an analytic vector field satisfies both condition $\omega$ and condition $A$, then it admits a convergent transformation to its normal form. It gives us the analytic part of Theorem \ref{thm:PrimitiveNF}.

In particular, in \textit{\textbf{the nonresonant case}}, if the the vector field satisfies the well-known Siegel's Diophantine condition, then the primitive form is analytically linearizable. Such analytic linearization results under Siegel's condition are already given for singular contact forms in  \cite{Zhito-1Form1992}, or for certain singular first order PDEs in \cite{Webster-1stOrder1987}.

\begin{remark}
A similar problem of simultaneous normalization of a pair $(\Pi,X)$, where $\Pi$ is a linear Poisson structure and $X$ is a vector field such that $\cL_X\Pi=\Pi$, was study in \cite{Mx_PoissonVF2006} by similar method.
\end{remark}

\subsection{Smooth normal form}
In this subsection, we pay attention to the generic case where the Liouville vector field $X$ is smooth ($C^\infty$) and hyperbolic, i.e., none of the eigenvalues of $dX(0)$ lies on the imaginary axis. By the classical Sternberg-Chen theorem, $X$ can be smoothly normalized to the Poincar\'e-Dulac normal form. 

However, the statement of Lemma \ref{lem2.1} does not hold any more if we directly replace formal vector fields by smooth hyperbolic vector fields, thus the argument of Lemma \ref{lem:Darboux} does not work neither.
Hence, the fact that $X$ is in the Poincar\'e-Dulac normal form does not directly imply the smooth normal form of the primitive form $\gamma$ as in the analytic case.
\begin{example}
  Consider the smooth hyperbolic vector field 
$$X=(2+3xy^2)x\frac{\partial}{\partial x}-(1+2xy^2)y\frac{\partial}{\partial y}$$ 
on $\mathbb R^2$. It is already in the Poincar\'e-Dulac normal form and it is a Liouville vector field of the standard symplectic form $\omega_0:=\mathrm dx\wedge\mathrm dy$ on $\mathbb R^2$. As $X$ is in the Siegel domain (i.e., $0$ is in the convex hull of the eigenvalues of the linear part of $X$), we can find a non-constant smooth function having vanishing infinite jet $j_0^\infty F$ at $0$ such that $X(F)=0$. Therefore, $X$ is a Liouville vector field of $\omega_1:=(1+F)\omega_0$ since $$\mathcal L_X \omega_1=X(1+F)\omega_0+(1+F)\mathcal L_X \omega_0=(1+F)\mathcal L_X \omega_0=\omega_1.$$ 
On the other hand, $X^s=2x\frac{\partial}{\partial x}-y\frac{\partial}{\partial y}$ is not a Liouville vector field of $\omega_1$. Otherwise, $X^s$ should preserve $(1+F)$ and therefore $F$ is a function of $xy^2$, which implies that $xy^2$ is a first integral of $X$. This gives the contradiction.
\end{example}

Fortunately, the smooth conjugacy for hyperbolic vector field is not unique due to the Sternberg-Chen theorem and we can find an alternative normalization which conjugates both the vector field and the symplectic form, by the following three assertions.

i) We can still use Moser's path method to conjugate the symplectic form $\omega$ to its constant part $\omega_0=\omega(0)$. Let $\phi_t$ be a family of diffeomorphisms such that $\phi_t^*\omega_t=\omega_0$ and $Y_t$ be the infinitesimal of $\phi_t$. Moreover, $Y_t$ can be chosen such that it formally commutes with the semi-simple part of the Liouville vector field $X$. Therefore, $\phi_1^*X$ is still formally in the Poincar\'e-Dulac normal form, still denoted by $X$ for simplicity of notations.

ii) We claim that there exsits a smooth Liouville vector field $\tilde X$ of $\omega_0$ having infinite contact with $X$ in formal normal form such that $\tilde X$ is genuinely in Poincar\'e-Dulac normal form. In fact, as $i_{X}\omega_0$ is a sum of its linear part $\beta_1$ and a differential form $\mathrm dH$ of some Hamiltonian function $H$ satisfying $X^s(H)=H$ in the formal sense, there exists a smooth function $\tilde H$ having infinite contact with $H$  by the Borel theorem such that $X^s(\tilde H)=\tilde H$. We define $\tilde X$ by $i_{\tilde X}\omega_0=\beta_1+\mathrm d\tilde H$. $\tilde X$ is a Liouville vector field and it is in Poincar\'e-Dulac normal form since
\begin{align*}
\cL_{\tilde X^s}\left(i_{\tilde X}\omega(0)\right)&=\cL_{\tilde X^s}(\beta_1+\mathrm d\tilde H)\\
&=\cL_{\tilde X^s}\beta_1+\mathrm d\cL_{\tilde X^s}\tilde H=\beta_1+\mathrm d\tilde H.
\end{align*}
In other hand, we also have
\begin{align*}
\cL_{\tilde X^s}\left(i_{\tilde X}\omega(0)\right)&=i_{[\tilde X^s,\tilde X^s]}\omega(0)+i_{\tilde X}\mathcal L_{\tilde X^s}\omega(0)\\
&=i_{[\tilde X^s,\tilde X]}\omega(0)+i_{\tilde X}\omega(0)=i_{[\tilde X^s,\tilde X]}\omega(0)+\beta_1+\mathrm d\tilde H.
\end{align*}

iii) We now apply a theorem in \cite{Manouchehri, Chaperon1999} saying that two smooth hyperbolic Liouville vector fields $X,\tilde X$ of a symplectic form $\omega$ are conjugate by a symplectic diffeomorphism if and only if they are formally conjugate. 

Here we give a direct proof. For any Liouville vector field $X$, the pullback $(\phi_X^t)^*\omega$ of $\omega$ by the flow $\phi_X^t$ of $X$ is just $\mathrm e^t\omega$, thus $(\phi_{\tilde X}^{-t}\circ\phi_X^t)^*\omega=\omega$ for all $t$. Notice that the conjugacy between $X$ and $\tilde X$ can be realized by a diffeomorphism of the form $\phi_{\tilde X}^{-t(p)}\circ\phi_X^{t(p)}(p)$ due to the Sternberg-Chen theorem. Hence the conjugacy does preserve the symplectic form.

We have completed the proof of the smooth part of theorem \ref{thm:PrimitiveNF}. We would like to point out that in \textbf{\textit{the nonresonant case}}, $X$ is automatically hyperbolic. By Sternberg's linearization theorem \cite{Ste}, we can assume that $X$ is linear. Therefore, we can linearize the primitive form $\gamma$.

\section{Normalization of singular contact forms}

\subsection{Transversal singularities}
We assume that $O$ is a transversal nondegenerate singularity. Then define a function $h$ on a small neighborhood $\mathcal{U}(O)$ of the origin $O$ as follows:
\begin{itemize}
\item{} $h=0$ on $N = \{x: \alpha \wedge (d\alpha)^n (x) = 0 \}$.
\item{} For $x\not\in N$, then $x$ lies on the integral curve $\Gamma$ of $Z$ from $y\in N$ to $x$, we put
\begin{equation}
h(x)=\int_{\Gamma}\alpha.
\end{equation}
\end{itemize} 
Then $h$ is a Morse function on each integral curve of $Z$. Since $h$ is transversally Morse (Morse-Bott), there is a coordinate function which we also denote by $\theta$ such that $\theta=0$ on $N$ and $h=\pm\theta^2/2$. Let us write the form $\alpha$ as
\begin{equation}
\alpha=\pm\theta d\theta+\gamma.
\end{equation} 
Since $\gamma(\frac{\partial}{\partial\theta})=(\alpha-dh)(\frac{\partial}{\partial\theta})=0$ and $d\beta(\frac{\partial}{\partial\theta})=d\alpha(\frac{\partial}{\partial\theta})=0$, $\gamma$  is a basic $1$-form, i.e. a $1$-form depends only on the coordinates $x_1,\ldots,x_{2n}$. Hence, the presymplectic form $d\alpha$ can be projected to the hyperplane $\{\theta=0\}$ and therefore $\gamma$ is indeed a primitive form  on the hyperplane. 

Thus we have finished the proof of the transversal part of theorem \ref{thm:prenormalization}. Together with Theorem \ref{thm:PrimitiveNF}, it gives us (formal, smooth or analytic) normalization of $\alpha$
in the transversal nondegenerate singular case. In particular, when the associated conformal vector field of $\gamma$ is non-resonant, we recover the following linear normalization for $\alpha$ (which is formal, or smooth or analytic provided that a Diophantine condition is satisfied), which was obtained before by Zhitomirskii 
\cite[Entry~10~of~Table~14.2]{Zhito-1Form1992}:
\begin{equation}
\alpha = \theta d \theta + \sum_{i=1}^{n}\left(\lambda_{n_i} x_{j}dx_{n+j}+(\lambda_{n_i}-1) x_{n+j}dx_{j}\right).
\end{equation}
The case when $\gamma$ is resonant was not studied in  
\cite{Zhito-1Form1992}. In this case, a priori the normal form for $\alpha$ must contain resonant nonlinear terms, and reads 
\begin{equation}
\alpha=\pm\theta d\theta+\gamma_0.
\end{equation} 
where $\gamma_0$ is a resonant normal form of $\gamma$ and takes one of the expressions presented in Theorem \ref{thm:PrimitiveNF}. 

\subsection{Tangent singularities}

In this subsection, we consider the case when $O$ is a tangent nondenegerate singularity of a smooth or analytic
singular contact 1-form $\alpha$.
We will assume that the tangency of 
$N = \{x: \alpha \wedge (d\alpha)^n (x) = 0 \}$ 
with the 1-dimensional 
kernel foliation of the presymplectic form $d\alpha$ at $O$ 
is generic, i.e. is of order $2$. We will now prove the tangent case of
the pre-normalization theorem \ref{thm:prenormalization} under this
generic tangency assumption. We will divide the pre-normalization
into four steps, and write one lemma for each step.

Using the Weierstrass preparation theorem in the analytic case, or the Malgrange preparation theorem in smooth case (see \cite{Malgrange_Ideals}), we can assume that
\begin{equation}
N=\{\theta^2=x_1\}
\end{equation}
in a local coordinate system $(\theta, x_1, \hdots, x_{2n})$
such that the kernel of $d\alpha$ is generated by 
$Z=\dfrac{\partial}{\partial\theta}$.
Then $\alpha$ can be written as
\begin{equation}
\alpha=(\theta^2-x_1)hd\theta+\sum_{i=1}^{2n}g_idx_i
\end{equation}
with $h(O)\neq 0$.

Let us restrict our attention to a plane 
$P = P_{c_2, \hdots,c_{2n}} = \{x_2=c_2,\ldots,x_{2n}=c_{2n}\}$, 
where $c_2,\ldots,c_{2n}$ are constant. For simplicity of notations,
redenote $x_1$ by $x$.

\begin{lemma}
Put $\beta=h(x,\theta)(\theta^2-x)d\theta$. There exists a smooth curve $\ell= \{\theta= x g(x) \}$ on $P$, where $g$ is a local smooth function, 
which is analytic in the analytic case,  such that
\begin{equation}
\label{eqn:ThetaCurve}
\int_{\theta = -\sqrt{x}}^{\theta =x g(x)}\beta=\int_{\theta = x g(x)}^{\theta = \sqrt{x}}\beta
\end{equation}
for all $x \geq 0$ sufficiently small in the real case
(for all $x$ near 0 in the complex holomorphic case), where the two
integrals are taken on the line $\{x\ \text{fixed}\}$.
\end{lemma}

\begin{figure}[!ht]
\includegraphics[width=0.6\textwidth]{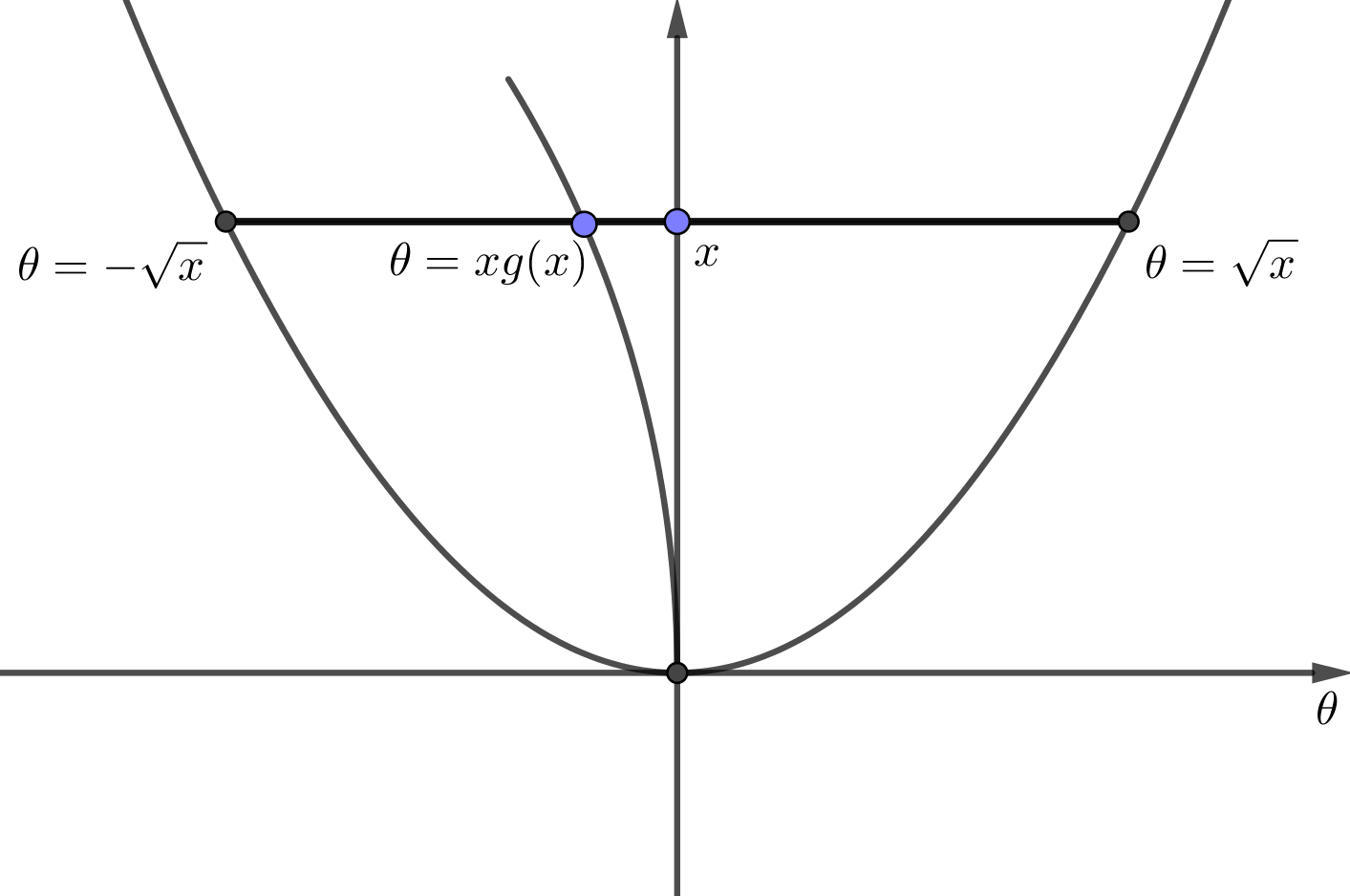}
\end{figure}

\begin{proof}
Denote
\begin{equation*}
F(x)=\int_{0}^{\sqrt{x}}\beta-\int_{-\sqrt{x}}^{0}\beta
\end{equation*}
(where the integrals are taken on the line $\{x\ \text{fixed}\}$).
We decompose $h(x,\theta)$ into the sum of odd and even functions on $\theta$:
\begin{equation}
h(x,\theta)=h_1(x,\theta^2)+\theta h_2(x,\theta^2).
\end{equation}
Then, $F$ can be written as
\begin{equation*}
F(x)=2\int_{0}^{\sqrt{x}}\theta h_2(x,\theta^2)(\theta^2-x)d\theta=\int_{0}^{x} h_2(x,\theta')(\theta'-x)d\theta'
\end{equation*}
(by the change of variable $\theta' = \theta^2$).
In particular, $F(x)$ is smooth. Moreover, we have
\begin{equation*}
\dfrac{dF(x)}{dx} =\int_{0}^{x} \frac{\partial h_2}{\partial x}(x,\theta')(\theta'-x)- h_2(\theta'-x)d\theta'.
\end{equation*} 
This implies that the derivative of $F$ at $0$ is $0$, therefore
we can write $F(x)=2x^2G(x)$, where $G(x)$ is smooth (analytic
in the analytic case).

Now, for any given function $y = g(x)$, put
\begin{equation*}
P(x,y)=\int_{0}^{xy} \beta=\int_{0}^{x y} h(x,\theta)(\theta^2-x)d\theta.
\end{equation*}
By taking $z = \theta/x \in[0,y]$ as the new variable 
for the integration, we get
\begin{equation*}
P(x,y)=x^2\int_{0}^{y} h(x,xz)(xz^2-1)dz.
\end{equation*}
Consider 
\begin{equation*}
S(x,y)=\int_{0}^{y} h(x,xz)(xz^2-1)dz-G(x).
\end{equation*}
We have $S(0,\frac{G(0)}{h(0,0)})=0$ and $\frac{\partial S}{\partial y}(0,\frac{G(0)}{h(0,0)})=-h(0,0)\neq 0$. By the implicit function theorem, there is a smooth function $g(x)$ in a neighborhood of $0$ such that $S(x,g(x))=0$ and $g(0)=\frac{G(0)}{h(0,0)}$. It implies that $F(x)=2P(x,g(x))$ (for all $x \geq0$ small enough in the real case, and all $x$ small enough in the holomorphic case). Equality \eqref{eqn:ThetaCurve} follows immediately from that. 
\end{proof}

Consider the diffeomorphism $\phi$ in a neighborhood of $(x,\theta)=(0,0)$, which moves the curve $\ell=\{\theta=xg(x)\}$ 
to the axis $\{\theta=0\}$, whose inverse map is defined as follows:
\begin{equation}
\phi^{-1}(x,\theta)=(x,xg(x)+\theta(1-\theta g(\theta^2))).
\end{equation}  
Doing it parameter-wise, for every plane $P_{c_2, \hdots,c_{2n}}$
we get a local diffeomorphism $\phi$ of the space which fixes the points of $N=\{\theta^2=x\}$, sends each curve $\ell$ to the axis $\theta=0$ on the corresponding plane $P_{c_2, \hdots,c_{2n}}$, and does not change
the direction of $\dfrac{\partial}{\partial \theta}$. After applying
this coordinate transformation $\phi$, we can assume that 
\begin{equation}
\int_{-\sqrt{x}}^{0}\beta=\int_{0}^{-\sqrt{x}}\beta
\end{equation}
on every plane $P_{c_2, \hdots,c_{2n}}$. 

Remark that, in the above formulas, one can replace $\beta$ by $\alpha$, the results will remain the same, just keep in mind that the integration is taken on the
curves generated by the kernel of $d\alpha$.

\begin{lemma} On each plane $P_{c_2, \hdots,c_{2n}}$
there is a new local coordinate system $(\eta,y)$, where $y=y(x)$,$\eta=\eta(x,\theta)$, such that $N=\{\eta^2=x\}$, $\ell=\{\eta=0\}$ and 
\begin{equation}
\label{eq1.5}
\int_{0}^{\sqrt{y}}\beta=\int_{0}^{\sqrt{y}}(\eta^2-y)d\eta
\end{equation}
(for all $y\geq 0$ sufficiently small in the real case, and all
$y$ near 0 in the complex case).
\end{lemma}

\begin{proof}
Equality \eqref{eq1.5} is equivalent to 
\begin{equation}\label{eq1.6}
\int_{0}^{\sqrt{x}}h(x,\theta)(\theta^2-x)d\theta=-\frac{2}{3}y^{3/2}.
\end{equation}
Let us decompose 
$$h(x,\theta)=h_1(x,\theta^2)+\theta h_2(x,\theta^2).$$
Since $\int_{-\sqrt{x}}^{0}\beta=\int_{0}^{\sqrt{x}}\beta$, we have
\begin{equation*}
\int_{0}^{\sqrt{x}}h(x,\theta)(\theta^2-x)d\theta=\int_{0}^{\sqrt{x}}h_1(x,\theta^2)(\theta^2-x)d\theta.
\end{equation*}
Denote by 
$$Q(z)=\int_{0}^{z}h_1(z^2,\theta^2)(\theta^2-z^2)d\theta.$$
We have $Q(-z)=-Q(z)$, $Q'(0)=0$ and $\left(\frac{Q'}{2z}\right)'(0)=-h(0,0)\neq 0$. It implies that $Q(z)=z^3(-h(0,0)+K(z^2))$ for some (smooth or analytic) function $K$. Therefore, the diffemorphism $\left( y(x)=x\left(\frac{3}{2}h(0,0)-\frac{3}{2}K(x)\right)^{2/3},
\eta(x,\theta) = \theta \left(\frac{3}{2}h(0,0)-\frac{3}{2}K(x)\right)^{1/3} \right)$ satisfies Equation \eqref{eq1.6}.
\end{proof}

Thus, we can assume, after a change of coordinates, that 
\begin{equation}\label{eq1.7}
\int_{0}^{\sqrt{x}}\beta=-\frac{2}{3}x^{3/2}, 
\end{equation}
where $\beta=h(x,\theta)(\theta^2-x)d\theta$
and $\alpha=\beta+\sum_{i=1}^{2n}g_idx_i$. From \eqref{eq1.7} it follows automatically that $h(0,0)=1$.

\begin{lemma}
With the above notations and assumptions, there is a local diffeomorphism $\phi(x,\theta)=(x,\psi(x,\theta))$ which fixes the points of $N=\{\theta^2=x\}$ and of $\ell=\{\theta=0\}$, and such that $\psi_0(\theta) := \psi(0,\theta) = \theta + \theta^3 L(\theta)$ 
(for some smooth function $L$) satisfies the 
following equation (for all $\theta$ sufficiently small):
\begin{equation}\label{eq1.8}
\int_{0}^{\theta} \beta_{|_{\{x=0\}}} := \int_{0}^{\theta}h(0,\eta)\eta^2d\eta
=\frac{(\psi_0(\theta))^3}{3}.
\end{equation}
\end{lemma}

\begin{proof}
Denote by $F(\theta)$ the left hand side of \eqref{eq1.8}. Then we have $F(0)=F'(0)=F''(0)=0$, $F'''(0)=2h(0,0)=2\neq 0$ and $F''''(0)=6\frac{\partial h}{\partial\theta}(0,0)$. By $\eqref{eq1.7}$, we have
\begin{equation}\label{eq1.9}
\int_{0}^x\theta h_2(x^2,\theta^2)(\theta^2-x^2)d\theta=0 \quad (\forall x,\ \text{or}\ \forall x\geq 0),
\end{equation}
where $h(x,\theta)=h_1(x,\theta^2)+\theta h_2(x,\theta^2)$. Denote by $G(x)$ the left hand side of $\eqref{eq1.9}$, then we have
\begin{equation*}
\frac{1}{x}\left(\frac{G'}{2x}\right)'(0)=-h_2(0,0).
\end{equation*}
It implies that $h_2(0,0)=0$. Therefore, $F$ can be written as 
$$F(\theta)=\theta^3\left(\frac{1}{3}+\theta^2H(\theta)\right).$$
Hence, we have a smooth function $\psi_0(\theta)=(3F(\theta))^{\frac{1}{3}}=\theta+\theta^3L(\theta)$. 
Now we need to find a smooth function
$\psi(x,\theta)=\psi_0(\theta)+x\theta\psi_1(x,\theta)$ such that the local diffeomorphism $\phi(x,\theta)=(x,\psi(x,\theta))$
fixes the points of $N=\{\theta^2=x\}$. We have 
$\psi(\theta^2,\theta)=\theta+\theta^3L(\theta)+\theta^3\psi_1(\theta^2,\theta)$, therefore it suffices to choose $\psi_1(x,\theta)=-L(\theta)$.
\end{proof}

\begin{lemma}
\label{lem:TangentPreNormal}
There exists a local smooth (or analytic) 
coordinate system $(x,\theta)$ in a neighborhood 
of $(0,0)$ such that
\begin{equation}
\beta=(\theta^2-x)d\theta.
\end{equation}
\end{lemma}

\begin{proof}
For the moment, we can assume that $\beta = h(x,\theta)(\theta^2-x)d\theta$ satisfies the conclusions of the previous lemmas, and that
\begin{equation}\label{eq1.8b}
\int_{0}^{\theta}h(0,\eta)\eta^2d\eta
=\frac{\theta^3}{3}
\end{equation}
(which is given by the last lemma, after a coordinate transformation).
We will show that there is a smooth function 
\begin{equation}
\xi(x,\theta) =\theta+x\nu(x,\theta)
\end{equation}
with $\dfrac{\partial\xi}{\partial\theta}(0,0)\neq 0$, 
such that
\begin{equation}\label{eq1.10}
\int_{0}^{\theta}h(x,\eta)(\eta^2-x)d\eta=\int_{0}^{\xi(x,\theta)}(\eta^2-x)d\eta=\frac{1}{3}\xi^3(x,\theta)-x\xi(x,\theta).
\end{equation}
Denote by $F(x,\theta)$ the left hand side of \eqref{eq1.10}. Then \eqref{eq1.10} is equivalent to
\begin{equation}\label{eq1.11}
\frac{1}{x}\left(F(x,\theta)-\frac{\theta^3}{3}\right)+\theta=(\theta^2-x)\nu+x\theta\nu^2+\frac{x^2\nu^3}{3} .
\end{equation}
Denote by $G(x,\theta)$ the left hand side of \eqref{eq1.11}. Since $F(0,\theta)=\dfrac{\theta^3}{3}$, the function $G(x,\theta)$ is smooth. We claim that $G(x,\theta)=(\theta^2-x)^2U(x,\theta)$, where $U$ is some smooth function. Indeed, let us denote by $y=\theta^2-x$ and $V(y,\theta)=G(\theta^2-y,\theta)$.  
By \eqref{eq1.7}, we have $V(0,\theta)=0$. Moreover, 
\begin{align*}
\frac{\partial V}{\partial y}(0,\theta)&=\frac{1}{\theta^4}\bigg(\theta^2\int_{0}^{\theta}-\frac{\partial h}{\partial x}(\theta^2,\eta)(\eta^2-\theta^2)
\\& \quad\qquad +h(\theta^2,\eta)d\eta+\int_{0}^{\theta}h(\theta^2,\eta)(\eta^2-\theta)d\eta-\frac{\theta^3}{3}\bigg)\\
&=-\frac{1}{2\theta^3}\frac{\partial}{\partial\theta}\left(\int_{0}^{\theta}h(\theta^2,\eta)(\eta^2-\theta^2)\right)-\frac{1}{\theta}\\
&=0, 
\end{align*}
which shows that $G(x,\theta)=(\theta^2-x)^2U(x,\theta)$. Let's write $\nu(x,\theta)=(\theta^2-x)\mu(x,\theta)$. Then \eqref{eq1.11} is equivalent to
$$U(x,\theta)=\mu+x\theta\mu^2+\frac{x^2(\theta^2-x)\mu^3}{3}.$$
The existence of $\mu$ is now due to the implicit function theorem, and $(x,\xi)$ (with $\xi$ playing the role of new $\theta$) is the desired new coordinate system which satisfies the conclusion of the lemma.
\end{proof}

After the above last step (Lemma \ref{lem:TangentPreNormal}), we get the following (smooth or analytic) pre-normal form for $\alpha$ near a 
generic tangent nondegenerate singularity:
\begin{equation} \label{eqn:alphaTangentPrenormal}
\alpha= (\theta^2 - x_1) d\theta +  \sum_{i=1}^{2n} g_i dx_i = d\left(\frac{\theta^3}{3}-\theta x_1\right)+  \gamma
\end{equation}
where
$\gamma =  (g_1 + \theta) dx_1 + \sum_{i=2}^{2n} g_i dx_i $
is a basic 1-form with respect to $\dfrac{\partial}{\partial\theta}$ (because the kernel of $d\gamma = d\alpha$ is generated by
$\dfrac{\partial}{\partial\theta}$). In other words, the functions $g_1 + \theta, g_2,\hdots, g_{2n}$ do not depend on $\theta$. Moreover, 
$\omega := d\gamma = d\alpha$ is a symplectic form in $2n$
variables $(x_1,\hdots, x_{2n})$ (once we forget about the variable $\theta$). Notice that the fractional coefficient $\dfrac{1}{3}$
in Formula \eqref{eqn:alphaTangentPrenormal} can be erased by a simple
linear change of variables. Thus we have proved the tangent part of
Theorem \ref{thm:prenormalization}. Remark that, via the induction
construction of a canonical system of coordinates of a symplectic form
(by choosing the coordinates one by one), we can assume that
$\omega = d\gamma = \sum_{i=1}^n dx_i \wedge dx_{n+i}$, where the function
$x_1$ already appears in the term $d (\theta^3- x_1 \theta)$ of 
the expression of $\alpha$ (i.e., we can take this given function as the first coordinate in a canonical system of coordinates for $\omega$).

The tangent part of Theorem \ref{thm:prenormalization} together with Theorem \ref{thm:PrimitiveNF} give us a normalization of $\alpha$ in the tangent nondegenerate singular case. In particular, similarly to the previous subsection, under the nonresonance assumption we can linearize $\gamma$ to get
the following normal form for $\alpha$:
\begin{equation}\label{eqn:LinearizationTangent}
\alpha=d(\theta^3-\phi(x_1,\ldots,x_{2n})\theta)+\sum_{i=1}^n\lambda_ix_idx_{n+i}+
\sum_{i=1}^n(\lambda_i -1) x_{n+i}dx_{i}.
\end{equation}

Notice that, in \eqref{eqn:LinearizationTangent}, the function $\phi$
is the old function $x_1$ in the pre-normal form of $\alpha$ before the
linearization (or normalization in general) of $\gamma$, but now it is not the new $x_1$ but a regular function $\phi(x_1,\ldots,x_{2n})$ which
cannot be linearized in general in the new variables $(x_1,\ldots,x_{2n})$,
contrary to the transversal nonresonant case where $\alpha$ can be fully
linearized. The reason is that $\gamma$ and $\phi$ are independent,
and a priori they cannot be simultaneously linearized. 

For example, consider the following very simple situation, with $n=1$:
\begin{equation}
\gamma = \dfrac{1}{2} x dy - \dfrac{1}{2} y dx, \quad \phi = x^2 + y.
\end{equation}
In this case, the associated conformal vector field is $X = \dfrac{1}{2} (x \dfrac{\partial}{\partial x} + y \dfrac{\partial}{\partial y})$ (in any linearizing system for $\gamma$). So if $\phi$ is to
be also linear in such a coordinate system, we must have that 
$X(\phi) = \phi/2$, but it is not the case here. In other words, $\phi$ can never be linear in a coordinate system in which $\gamma$ is also linear.

\vspace{0.5cm}

\section*{Acknowledgements}

Kai Jiang and Nguyen Tien Zung would like to thank the Center for
Geometry and Physics, Institute for Basic Science (South Korea) for the
invitation to the Center and for excellent working conditions, which enabled
us to work together intensively on some parts of this joint research 
project in April 2018.

\end{document}